\documentclass[reqno]{amsart}
\usepackage{amsmath, amssymb, amsthm, epsfig}
\usepackage{hyperref, latexsym}
\usepackage{url}
\usepackage{color}
\usepackage{fullpage}
\usepackage{setspace}
\usepackage{graphics}

\allowdisplaybreaks

\def\today{\ifcase\month\or
	January\or February\or March\or April\or May\or June\or
	July\or August\or September\or October\or November\or December\fi
	\space\number\day, \number\year}

\newtheorem{theorem}{Theorem}

\newtheorem{lemma}[theorem]{Lemma}

\theoremstyle{definition}

\theoremstyle{remark}

\newcommand{\de}[0]{\mathrel{\mathop:}=}
\newcommand{\ie}[0]{\mathrm{i}}

\newcommand{\R}{\mathbb{R}}
\newcommand{\N}{\mathbb{N}}

\newcommand{\Z}{\mathbb{Z}}

\newcommand{\ds}{\text{\rm d}s}

\renewcommand{\d}{\text{\rm d}}
\newcommand{\du}{\text{\rm d}u}

\newcommand{\im}{{\rm Im}\,}
\newcommand{\re}{{\rm Re}\,}

\begin{document}
	
	\title[Explicit conditional bounds for $\zeta(s)$  \\ at the edge of the critical strip]{Explicit conditional bounds for $\zeta(s)$ \\ at the edge of the critical strip}
	\author[Andrés~Chirre and Blas~Molero]{Andrés~Chirre and Blas Molero}
	\subjclass[2010]{11M06, 11M26, 41A30}
	\keywords{Riemann zeta-function, Riemann hypothesis, bandlimited functions}
	\address{Departamento de Ciencias - Sección Matemáticas, Pontificia Universidad Católica del Perú, Av. Universitaria 1801, San Miguel 15088, Lima, Perú}
	\email{cchirre@pucp.edu.pe}
	\address{Departamento de Ciencias - Sección Matemáticas, Pontificia Universidad Católica del Perú, Av. Universitaria 1801, San Miguel 15088, Lima, Perú}
	\email{blas.molero@pucp.edu.pe}
	\allowdisplaybreaks
	\numberwithin{equation}{section}
	
	\begin{abstract}
		In this paper, we obtain explicit bounds for the real part of the logarithmic derivative of the Riemann zeta-function on the line $\re s=1$, assuming the Riemann hypothesis. The proof combines the Guinand--Weil explicit formula with extremal bandlimited
		majorants and minorants for the Poisson kernel. As an application, we revisit the classical estimates of Littlewood for the modulus of the Riemann zeta-function and of its reciprocal on the line $\re{s}=1$, and derive a slight refinement of the bounds of Lamzouri, Li, and Soundararajan. In addition, we establish an explicit bound for the modulus of the logarithmic derivative of the Riemann zeta-function on the line $\re{s}=1$ under the Riemann hypothesis, improving the lower-order term in a result of Chirre, Hagen, and Simoni\v{c}.
	\end{abstract}
	\maketitle
	\thispagestyle{empty}
	
	\section{Introduction}
	A central theme in analytic number theory is to make the influence of the zeros of the Riemann zeta-function on arithmetic phenomena quantitatively explicit. The line $\re s=1$ is a particularly delicate region, where the interaction between sums over prime numbers and the distribution of zeros controls the size of the Riemann zeta-function and its logarithmic derivative. In this setting, we develop an explicit analysis and obtain new bounds that refine the known asymptotic behavior by yielding improved lower–order terms.
	
	Let $\zeta(s)$ be the Riemann zeta function. Assuming the Riemann hypothesis (RH), classical estimates due to Littlewood \cite{Li1, Li2} assert that, as $t \to \infty$,
	\begin{equation*}
		|\zeta(1+it)|\leq (2e^\gamma+o(1))\log\log t,  \,\,\,\,\,\,\,\,\,\,\,\,\,\,\, \mbox{and}  \,\,\,\,\,\,\,\,\,\,\,\,\,\,\,\,
		\dfrac{1}{|\zeta(1+it)|}\leq \left(\dfrac{12e^\gamma}{\pi^2}+o(1)\right)\log\log t,
	\end{equation*}  
	where $\gamma=0.5772\ldots$ is the Euler-Mascheroni constant. These asymptotic behaviors have remained unimproved for a century, with efforts instead focused on sharpening lower-order error terms. The sharpest explicit versions of these results were obtained by Lamzouri, Li, and Soundararajan \cite[p.~2394]{Sound}, who proved that, under RH and for $t\geq 10^{10}$,
	\begin{equation} \label{11_52am}
		|\zeta(1+it)|\leq 2e^\gamma\left(\log\log t - \log2+ \frac{1}{2}  +\frac{1}{\log\log t}\right),
	\end{equation}  
	and 
	\begin{align} \label{11_53am}
		\dfrac{1}{|\zeta(1+it)|}\leq \dfrac{12e^\gamma}{\pi^2}\left(\log\log t - \log2+ \frac{1}{2}  +\frac{1}{\log\log t}+\frac{14\log\log t}{\log t}\right).
	\end{align}
	On the other hand, the classical conditional bound for $\zeta'(s)/\zeta(s)$ at $s=1+it$ is $O(\log\log t)$, see \cite[Corollary~13.14]{MV} for example. The sharpest explicit version of this result is due to Chirre, Hagen, and Simoni\v{c} \cite[Theorem~5]{ChiSim}, who proved that, under RH and for $t\ge 10^{30}$,
	\begin{align} \label{12_32pm}
		\left|\frac{\zeta'}{\zeta}\left(1+\ie t    \right)\right| \leq 2\log\log t -0.4989 + \dfrac{5.35(\log\log t)^2}{\log t}. 
	\end{align}
	These bounds have also been investigated in the context of families of $L$-functions, where several generalizations are known, see \cite{Langu,Lum,SimonicPalo}.
	
	The aim of this paper is to refine estimates \eqref{11_52am} and \eqref{11_53am} by improving the constants in the lower-order terms, and to improve the order of magnitude of the lower-order error term in \eqref{12_32pm}.

	\subsection{Main results} 
	We now present our main results. All of them are conditional on the Riemann hypothesis and hold for $
	t \ge e^{18} = 6.565\ldots \cdot 10^{7}$.
	Our first result concerns bounds for the real part of the logarithmic derivative of the Riemann zeta-function on the line $\re{s}=1$.
	

	\begin{theorem}
		\label{thm:Rezeta'/zeta}
		Assume the Riemann hypothesis. Then, for $t\geq e^{18}$, 
		\begin{align*} 
			\re\frac{\zeta'}{\zeta}\left(1+\ie t\right) \leq 2\log\log t +1-\gamma-\log4 +\frac{8\log\log t}{\log t}-\dfrac{8.6544}{\log t},
		\end{align*}
		and
		\begin{align*} 
			-\re\frac{\zeta'}{\zeta}\left(1+\ie t\right) \leq 2\log\log t +1-\gamma-\log4 - \frac{8\log\log t}{\log t}+\dfrac{6.9856}{\log t}.
		\end{align*}
	\end{theorem}
		
	As an application of Theorem \ref{thm:Rezeta'/zeta}, and combining ideas from \cite{Sound}, we obtain the following estimates.
	
	\begin{theorem}
		\label{thm:zeta}
		Assume the Riemann hypothesis. Then, for $t\geq e^{18}$,
		\begin{align} \label{3pm}
			|\zeta(1+it)|\leq 2e^\gamma\left(\log\log t - \log2+ \frac{1}{2}  +\frac{0.2674}{\log\log t}-\dfrac{2.6\log\log t}{\log t}\right),
		\end{align}
		and
		\begin{align} \label{3pm2}
			\left|\dfrac{1}{\zeta(1+it)}\right| \leq\frac{12e^\gamma}{\pi^2}\bigg(\log\log t -\log 2 +\frac{1}{2}+\frac{5}{8\log\log t} +\frac{10.8}{(\log\log t)^2}\bigg).
		\end{align}
	\end{theorem}
	Our bounds \eqref{3pm} and \eqref{3pm2} improve \eqref{11_52am} and \eqref{11_53am}, respectively, by reducing the constant in the numerator of the term $1/\log\log t$; moreover, in the upper bound we introduce a new negative lower-order term. The mentioned numerator in \eqref{3pm} comes from $$\min_{0<c<2}\frac{1}{8}\left(\frac{2}{c}+(2\log c +\log 4 -1)^2\right)=0.2673\ldots.$$
	
	\vspace{0.2cm}
	
	In our next application of Theorem \ref{thm:Rezeta'/zeta}, we bound the modulus of the logarithmic derivative of the Riemann zeta-function using Selberg’s moment formula, following \cite{ChiSim} but estimating the sum over the primes in a different manner.
	
	\begin{theorem} \label{thm:zeta'/zeta}
		Assume the Riemann hypothesis. Then, for $t\geq e^{18}$,
		\begin{equation*}	
			\left|\frac{\zeta'}{\zeta}(1+it)\right|  \leq   2\log\log t  + 0.0784 - \gamma +\dfrac{9.0581\log\log t}{\log t} - \dfrac{4.7}{\log t}.
		\end{equation*} 
		In particular, for $t\geq 10^{30}$ we have $|({\zeta'}/{\zeta})(1+it)|\leq   2\log\log t$. 
	\end{theorem}
	The main novelty here is an improvement over \eqref{12_32pm}, reducing the numerator of the lower-order term from $(\log\log t)^2$ to $\log\log t$, and adding a negative term. The constant term there is
	$$
	\min_{\lambda>0}\frac{e^\lambda+1}{2\lambda}-\lambda -\gamma = 0.0783\ldots-\gamma =-0.4989\ldots.
	$$





		
		\subsection{Notation} Throughout the paper, we use $\alpha=O^*(\beta)$ to mean that $|\alpha|\leq \beta$. For a function $f\in L^1(\R)$ we define its Fourier transform by $$\widehat{f}(\xi)=\int_{-\infty}^\infty f(x)e^{-2\pi i\xi x}\d x,
		$$
		where $\xi\in\R$. The symbol $\Gamma(s)$ denotes the Gamma function, and $\Lambda(n)$ denotes the von Mangoldt function defined by $\log p$ if $n=p^m$, where $p$ a prime number and $m\geq 1$ is an integer, and to be zero otherwise. The symbol $\rho$ denotes a non-trivial zero of $\zeta(s)$, and
		\begin{align} \label{3_35pm}
			B= -\sum_{\rho}\re\dfrac{1}{\rho} = \dfrac{\log 4\pi}{2}-1-\dfrac{\gamma}{2} = -0.02309\ldots,
		\end{align}
		see \cite[Corollary 10.14]{MV}. In particular, under RH we have $\sum_{\rho}\frac{1}{|\rho|^2}=2|B|$.


		\section{Preliminary lemmas}
		
		We begin with a representation lemma connecting the real part of the logarithmic derivative of the Riemann zeta-function and the Poisson kernel defined by
		\begin{align} \label{3_36pm}
			h(x)=\dfrac{\frac{1}{2}}{\frac{1}{4}+x^2}.
		\end{align}
		
		\begin{lemma} \label{4_04pm}
			Let $h$ be the function defined in \eqref{3_36pm}, and assume the Riemann hypothesis. Then, for $t>0$,
			\begin{equation*} 
				\re\dfrac{\zeta'}{\zeta}(1+it) = \displaystyle\sum_{\gamma}h(t-\gamma) - \frac{1}{2}\log\left(\frac{t}{2\pi}\right) +O^*\left(\dfrac{7}{4t^2}\right),
			\end{equation*}
			where the sum is taken over the ordinates $\gamma$ of the non-trivial zeros $\rho = \tfrac12 + i \gamma$ of $\zeta(s)$.
		\end{lemma}
		\begin{proof} We begin with the partial fraction decomposition of ${\zeta'}(s)/{\zeta}(s)$ (see \cite[Corollary 10.14]{MV}),
			\begin{equation*}
				\dfrac{\zeta'}{\zeta}(s) = \displaystyle\sum_{\rho}\bigg(\dfrac{1}{s-\rho}+\frac{1}{\rho}\bigg)-\dfrac{1}{2}\dfrac{\Gamma'}{\Gamma}\bigg(\dfrac{s}{2}\bigg)+B+\frac{\log\pi}{2} - \dfrac{1}{s}-\frac{1}{s-1},
			\end{equation*}
			where the sum runs over the non-trivial zeros $\rho$ of $\zeta(s)$, and $B$ is defined in \eqref{3_35pm}. Setting $s=1+it$, taking real parts, and using RH, we arrive at
			\begin{equation} \label{6_39pm}
				\re\dfrac{\zeta'}{\zeta}(1+it) = \displaystyle\sum_{\gamma}h(t-\gamma) - \frac{1}{2}\re\dfrac{\Gamma'}{\Gamma}\left(\frac{1}{2}+\frac{it}{2}\right) + \frac{\log\pi}{2} +O^*\left(\dfrac{1}{t^2}\right).
			\end{equation}
			By Stirling's formula (\cite[Lemma 3.11, p. 67]{HaraldG}) we have
			\begin{align} \label{3_53pm2}
				\re\frac{\Gamma'}{\Gamma}(z)= \log |z| -
				\dfrac{1}{2}\re\frac{1}{z}+ O^*\left(\dfrac{1}{4|z|^2}\right)\!, \,\,\,\,\,\mbox{for}\,\,\,\, \re{z}> 0.
			\end{align} 
			This shows that
			\begin{align} \label{6_36pm}
				\re\frac{\Gamma'}{\Gamma}\left(\frac{1}{2}+\frac{it}{2}\right)=\log\left(\frac{t}{2}\right) + O^*\left(\frac{3}{2t^2}\right).
			\end{align}
			Substituting this into \eqref{6_39pm} we obtain the desired result.
		\end{proof}
		
		\vspace{0.15cm}

		This representation naturally leads to the analysis of sums over the zeros of the Riemann zeta-function. The classical tool for estimating such sums is the Guinand–Weil explicit formula for the Riemann zeta-function~\cite[Lemma 8]{CChiM}. We state here a conditional version of this formula, which will be useful for our purposes.
		
		\begin{lemma}
			\label{Guinand-weilDirichlet}
			Let $g(s)$ be analytic in the strip $\left|\im{s}\right|\leq \tfrac12+\varepsilon$ for some $\varepsilon>0$, and assume that $|g(s)|\ll(1+|s|)^{-(1+\delta)}$ as $\left|\re s\right|\to\infty$, for some $\delta>0$. Assume the Riemann hypothesis. Then, \begin{align} \label{11_10pm2}
				\begin{split}
				\displaystyle\sum_{\gamma}g(\gamma) &  = \dfrac{1}{2\pi}\int_{-\infty}^{\infty}g(u)\,\re{\dfrac{\Gamma'}{\Gamma}\left(\dfrac{1}{4}+\dfrac{iu}{2}\right)}\,\du - \dfrac{\log\pi}{2\pi}\widehat{g}(0) + g\left(\dfrac{1}{2i}\right) + g\left(-\dfrac{1}{2i}\right) \\
				&  \ \ \ \ \ \ \ \ \ \ \ \ \ -\dfrac{1}{2\pi}\displaystyle\sum_{n\geq2}\dfrac{\Lambda(n)}{\sqrt{n}}\left(\widehat{g}\left(\dfrac{\log n}{2\pi}\right)+\widehat{g}\left(\dfrac{-\log n}{2\pi}\right)\right).
				\end{split}
			\end{align}
		\end{lemma}
		
		Clearly, the function $h$ defined in \eqref{3_36pm} does not satisfy the conditions of Lemma \ref{Guinand-weilDirichlet}. Then, the crucial step is to replace the Poisson kernel $h$ in Lemma \ref{4_04pm} with explicit bandlimited majorants and minorants (given in Section \ref{2_01pm}), that are compatible with the Guinand-Weil explicit formula.
		
		This approach, based on the Guinand–Weil explicit formula with carefully chosen bandlimited functions, has proven to be a powerful tool for deriving bounds on a wide range of quantities associated with the Riemann zeta-function. It allows one to balance the contributions from the zeros and the prime powers, leading to effective estimates in several problems of analytic number theory. See, for instance \cite{CChaM, CChiM, CMi, CS, chirreGoncalves, GG}.

		\section{The extremal functions and properties} \label{2_01pm}
			Let $h$ be the function defined in \eqref{3_36pm}. In~\cite[Lemma 9]{CChiM}, Carneiro, Chirre, and Milinovich proved that for any $\Delta>0$ the functions
		\begin{align*}  
			h^{\pm}_{\Delta}(z)=\left(\dfrac{\frac{1}{2}}{\frac{1}{4}+z^2}\right)\left(\dfrac{e^{\pi\Delta}+e^{-\pi\Delta}-2\cos(2\pi\Delta z)}{\left(e^{\pi\Delta/2}\mp e^{-\pi\Delta/2}\right)^2}\right)
		\end{align*}
		are real entire functions of exponential type\footnote{The entire function $f$ is said to be of exponential type $2\pi\Delta$ if $\limsup_{|z|\to\infty}|z|^{-1}\log|f(z)|=2\pi\Delta$.} $2\pi\Delta$ that satisfy the following properties:
		\begin{equation} \label{I)}
0\leq h^{-}_{\Delta}(x)\leq h(x)\leq h^{+}_{\Delta}(x)\, \,\, \mbox{for all}\,\, x\in \R,
\end{equation}
and 
	\begin{equation}  \label{II)}
 \widehat{h^{\pm}_{\Delta}}(\xi)=\mathrm{1}_{[-\Delta,\Delta]}(\xi)\cdot {\pi}\left(\dfrac{e^{\pi(\Delta-|\xi|)}-e^{-\pi(\Delta-|\xi|)}}{c^{\pm}_\Delta}\right)\,\, \mbox{for all}\,\, \xi\in \R, \mbox{where}\,\, c^{\pm}_\Delta=
 \big(e^{\pi\Delta/2}\mp e^{-\pi\Delta/2}\big)^{2}.
 \end{equation} 

		Moreover, the functions $h^{\pm}_{\Delta}$ are optimal in the sense that they minimize the $L^1$-distance among all functions of exponential type $2\pi\Delta$ that act as majorants or minorants. The Beurling–Selberg extremal problem in approximation theory seeks one-sided bandlimited approximations of a given function $f:\R\to\R$ to achieve minimal $L^1$-distance. Consequently, $h_{\Delta}^{\pm}$ provide the solution to the Beurling-Selberg extremal problem for the Poisson kernel defined in \eqref{3_36pm}. This construction follows from the general Gaussian subordination framework developed by Carneiro, Littmann, and Vaaler \cite{CLV}.
		
		\vspace{0.14cm}
		
		We proceed to estimate each term in the Guinand-Weil explicit formula (Lemma \ref{Guinand-weilDirichlet}) using the corresponding majorants and minorants. The next lemma provides an explicit asymptotic expansion of the Gamma term in \eqref{11_10pm2}. In previous works this term is usually treated only through its main contribution, up to a bounded error $O^*(1)$. Here we derive a more refined expansion by retaining further lower–order terms.
		
		
		
		\begin{lemma} \label{1_57pm}
			Let $\Delta\geq \frac{1}{2}$ and $t> 0$.
			Then
			\begin{align*}
				\dfrac{1}{2\pi}\int_{-\infty}^{\infty}h^{\pm}_\Delta(t-u)\,\re\,\dfrac{\Gamma'}{\Gamma}\bigg(\frac{1}{4}+\dfrac{iu}{2}\bigg) \,\du  =
				\bigg(\frac{1}{2} +\frac{1}{\pm e^{\pi\Delta}-1} \bigg)\log\left(\frac{t}{2}\right) +O^*\left( \dfrac{1.3}{t^2}\right).
			\end{align*}
		\end{lemma}
		\begin{proof} We begin by writing $h^{\pm}_\Delta(z)= \big(H_\Delta(z)+H_\Delta(-z)\big)/c^{\pm}_\Delta$, where $
			H_{\Delta}(z)=h(z)({\cosh(\pi\Delta)-e^{2\pi i\Delta z}}).
			$
			Therefore,
			\begin{align*}
				\dfrac{1}{2\pi}\int_{-\infty}^{\infty}h^{\pm}_\Delta(t-u)\dfrac{\Gamma'}{\Gamma}\bigg(\frac{1}{4}+\dfrac{iu}{2}\bigg) \,\du   =  \dfrac{1}{2\pi c^{\pm}_\Delta}\int_{-\infty}^{\infty}H_\Delta(u)\left(\,\dfrac{\Gamma'}{\Gamma}\bigg(\frac{1}{4}+\dfrac{i(t-u)}{2}\bigg)+\dfrac{\Gamma'}{\Gamma}\bigg(\frac{1}{4}+\dfrac{i(t+u)}{2}\bigg)\right)\du.
			\end{align*}
			Let $I$ denote the right-hand side of the above expression. To compute $I$, we apply the residue theorem using a contour given by a semicircle $\mathcal{C}_R$ with radius $R>0$ in the closure of the upper half-plane $\mathbb{H}^+$. Observe that $H_\Delta(z)$ has a simple pole in $\mathbb{H}^+$ at $z=\frac{i}{2}$, whereas $\frac{\Gamma'}{\Gamma}\big(\frac{1}{4}+\frac{i(t-z)}{2}\big)+\frac{\Gamma'}{\Gamma}\big(\frac{1}{4}+\frac{i(t+z)}{2}\big)$ has simple poles at $z=-t + 2(\ell+\frac{1}{4})i$ in $\mathbb{H}^+$, where $\ell\in\Z_{\geq 0}$ (they come from the second summand). It is straightforward that $|H_\Delta(z)|\ll R^{-2}$ for $z\in\mathcal{C}_R$. On the other hand, one can prove\footnote{Clearly, when $\Re{s}\geq 0$ it follows by Stirling's formula. When $\Re{s}<0$, we use the reflection formula $\frac{\Gamma'}{\Gamma}(s)= \frac{\Gamma'}{\Gamma}(1-s) - \pi\cot(\pi s)$, and the fact that $\cot(\pi s)$ is periodic and bounded when $s$ stays at distance $\delta>0$ from each $n\in \mathbb{Z}_{\leq 0}$.} that $\frac{\Gamma'}{\Gamma}(s)\ll_\delta\log|s|$ when $s$ stays at distance $\delta>0$ from the poles of $\Gamma(s)$. Hence, we may choose a sequence of radii $R>0$, with $R\to \infty$, such that $\frac{\Gamma'}{\Gamma}\big(\frac{1}{4}+\frac{i(t-z)}{2}\big)+\frac{\Gamma'}{\Gamma}\big(\frac{1}{4}+\frac{i(t+z)}{2}\big)\ll \log R$, for $z\in \mathcal{C}_R$. Accordingly, letting $\mathcal{P}(\Delta,t)$ denote the set of poles $\{\frac{i}{2}, -t +2(\ell +\frac{1}{4})i, \ell\in \mathbb{Z}_{\geq 0}\}$ (which are simple), we see that
			\begin{align} \label{5_38pm}
				\begin{split}
					{I}  & =  \dfrac{i}{c^{\pm}_\Delta}\sum_{z\in \mathcal{P}(\Delta,t)} \operatorname*{Res}\limits_{s=z} H_\Delta(s)\left(\,\dfrac{\Gamma'}{\Gamma}\bigg(\frac{1}{4}+\dfrac{i(t-s)}{2}\bigg)+\dfrac{\Gamma'}{\Gamma}\bigg(\frac{1}{4}+\dfrac{i(t+s)}{2}\bigg)\right) \\
					&  =\dfrac{\sinh(\pi\Delta)}{2c^{\pm}_\Delta}\left(\dfrac{\Gamma'}{\Gamma}\bigg(\dfrac{1}{2}+\dfrac{it}{2}\bigg)+\dfrac{\Gamma'}{\Gamma}\bigg(\dfrac{it}{2}\bigg)\right) - \dfrac{2}{c^{\pm}_\Delta} \sum_{\ell=0}^\infty H_\Delta(-t+2(\ell+\tfrac{1}{4})i).
				\end{split}
			\end{align}
			Using the identity $\frac{\Gamma'}{\Gamma}(s_1)-\frac{\Gamma'}{\Gamma}(s_2)=\sum_{\ell=0}^\infty\big(\tfrac{1}{\ell+s_2}-\tfrac{1}{\ell+s_1}\big)$ for $s_1=\frac{1}{2}+\frac{it}{2}$ and $s_2=\frac{it}{2}$, the series $\sum_{\ell=0}^\infty H_\Delta(-t+2(\ell+\frac{1}{4})i)$ can be written as
			\begin{align*}
				\frac{\cosh(\pi\Delta)}{4}\left(\dfrac{\Gamma'}{\Gamma}\bigg(\dfrac{it}{2}\bigg)-\dfrac{\Gamma'}{\Gamma}\bigg(\dfrac{1}{2}+\dfrac{it}{2}\bigg)\right)+O^*\left(\frac{1}{2}\sum_{\ell=0}^\infty\frac{e^{-4\pi\Delta(\ell+\frac{1}{4})}}{|\frac{1}{4}+(-t+2(\ell+\frac{1}{4})i)^2|}\right).
			\end{align*}
			One has $|\frac{1}{4}+(-t+2(\ell+\frac{1}{4})i)^2|\geq t^2$, which implies that the error term in the above expression is $O^*\big(\frac{e^{-\pi\Delta}}{2(1-e^{-4\pi\Delta})t^2}\big)$. Inserting this into \eqref{5_38pm}, rearranging the terms, and taking real parts we obtain
			\begin{align*} 
				\re I = & \dfrac{\sinh(\pi\Delta)}{c^{\pm}_\Delta}\re\dfrac{\Gamma'}{\Gamma}\bigg(\dfrac{1}{2}+\dfrac{it}{2}\bigg) + \dfrac{e^{-\pi\Delta}}{2c^{\pm}_\Delta}\bigg(\re\dfrac{\Gamma'}{\Gamma}\bigg(\dfrac{1}{2}+\dfrac{it}{2}\bigg)-\re\dfrac{\Gamma'}{\Gamma}\bigg(\dfrac{it}{2}\bigg)\bigg) +O^*\left(\frac{e^{-\pi\Delta}}{(1-e^{-4\pi\Delta})c^{\pm}_\Delta t^2}\right).
			\end{align*}
			By \eqref{3_53pm2}, $\re\frac{\Gamma'}{\Gamma}\left(\frac{1}{2}+\frac{it}{2}\right)-\re\frac{\Gamma'}{\Gamma}\left(\frac{it}{2}\right)=\frac{1}{2}\log(1+\frac{1}{t^2})-\frac{1}{1+t^2} +O^*(\frac{2}{t^2})$ which is $O^*(\frac{5}{2t^2})$. Finally, upon inserting \eqref{6_36pm} into the above expression, and using that 
			$$\,\,\,\,\,\,\,\,\,\,\,\,\,\,\,\,\,\,\,\,\,\,\,\frac{\sinh(\pi\Delta)}{c^{\pm}_\Delta}=\frac{1}{2} +\frac{1}{\pm e^{\pi\Delta}-1}, \,\,\,\,\,\,\,\,\,\,\, \mbox{and} \,\,\,\,\,\,\,\,\,\,\,\,\,\, \frac{3}{2}\dfrac{\sinh(\pi\Delta)}{c^{\pm}_\Delta}+\dfrac{e^{-\pi\Delta}}{c^{\pm}_\Delta}\left(\frac{5}{4}+\frac{1}{1-e^{-4\pi\Delta}}\right)\leq 1.298\ldots,$$
			for $\Delta \geq \frac{1}{2}$, the lemma follows. 
		\end{proof}
		
		\begin{lemma} \label{1_57pm2} Let $\Delta>0$ and assume the Riemann hypothesis. Then
			$$
			\dfrac{1}{\pi}\sum_{n\geq 2}\dfrac{\Lambda(n)}{\sqrt{n}} \left|\widehat{h^{\pm}_{\Delta}}\left(\dfrac{\log n}{2\pi}\right)\right| \leq \dfrac{ e^{\pi\Delta}}{c^{\pm}_\Delta}\left(2\pi\Delta - (1+\gamma) + \dfrac{\log 2\pi}{e^{2\pi\Delta}} + \dfrac{2|B|}{e^{\pi\Delta}}-\frac{1}{6e^{6\pi\Delta}}\right),$$
			where $B$ is defined in \eqref{3_35pm}.
		\end{lemma}
		\begin{proof}
			By \eqref{II)}, 
			\begin{align*}
				\sum_{n\geq 2}\dfrac{\Lambda(n)}{\sqrt{n}} \left|\widehat{h^{\pm}_{\Delta}}\left(\dfrac{\log n}{2\pi}\right)\right| = \sum_{n\leq e^{2\pi\Delta}}\dfrac{\Lambda(n)}{\sqrt{n}} \widehat{h^{\pm}_{\Delta}}\left(\dfrac{\log n}{2\pi}\right) =\pi \dfrac{e^{\pi\Delta}}{c^{\pm}_\Delta} \sum_{n\leq e^{2\pi\Delta}}\dfrac{\Lambda(n)}{n}\left(1-\dfrac{n}{e^{2\pi\Delta}}\right).
			\end{align*}
			The result follows by applying \cite[Lemma 2.4]{Sound}, see also Lemma \ref{12_01pm2}. 
		\end{proof}

		\begin{lemma} \label{1_57pm3}
			Let $\Delta\geq \frac{1}{2}$ and $t\in \R$. Then
			$$
			\left|h^{\pm}_\Delta\left(t-\frac{1}{2i}\right)+h^{\pm}_\Delta\left(t+\frac{1}{2i}\right)\right|\leq \dfrac{3.4}{t^2}. 
			$$
		\end{lemma}
		\begin{proof}
			From \cite[Eq. (3.21)]{CChiM} with $\beta=\frac{1}{2}$ it follows that
			\begin{equation*}
				h^{\pm}_\Delta(z) = \dfrac{2}{c^\pm_\Delta}\left(\frac{\sin(\pi \Delta (z + {1}/{2i}))}{z + {1}/{2i}} \right)\left(\frac{\sin(\pi \Delta (z - {1}/{2i}))}{z - {1}/{2i}} \right).
			\end{equation*}
			Consequently,
			\begin{equation*}
				h^{\pm}_\Delta\left(t-\frac{1}{2i}\right)+h^{\pm}_\Delta\left(t+\frac{1}{2i}\right) = \dfrac{4}{c^\pm_\Delta} \left(\frac{\sin{(\pi\Delta t)}}{t}\right)\left(\re\frac{\sin{(\pi\Delta (t-i))}}{t-i}\right)=O^*\bigg(\dfrac{2\,(e^{\pi\Delta} + e^{-\pi\Delta})}{c^{\pm}_\Delta t^2}\bigg).
			\end{equation*}
			Since $2\,(e^{\pi\Delta} + e^{-\pi\Delta})/c^\pm_\Delta\leq 3.325\ldots$ for $\Delta\geq \frac{1}{2}$, the claim follows.
		\end{proof}

		\vspace{0.2cm}
		
		\section{Proof of Theorem \ref{thm:Rezeta'/zeta}}
		By Lemma \ref{4_04pm} and \eqref{I)} we have for $t>0$ and $\Delta>0$ that
		\begin{equation} \label{10_23am}   
			\pm \re\dfrac{\zeta'}{\zeta}(1+it)\leq \pm \left(\sum_{\gamma}h_\Delta^{\pm}(t-\gamma) - \frac{1}{2}\log\left(\frac{t}{2\pi}\right)\right)+\dfrac{7}{4t^2}.
		\end{equation} 
		For a fixed $t$, the function $z\mapsto h^{\pm}_\Delta(t-z)$ satisfies the hypotheses of the Guinand-Weil explicit formula (see \cite[p. 634]{CChiM}). Hence, recalling that $h^{\pm}_\Delta$ are even, 
		\begin{align*}
			\displaystyle\sum_{\gamma}h^{\pm}_\Delta(t-\gamma) & = 
			\dfrac{1}{2\pi}\int_{-\infty}^{\infty}h^{\pm}_{\Delta}(t-u)\,\re\,\dfrac{\Gamma'}{\Gamma}\bigg(\dfrac{1}{4}+\dfrac{iu}{2}\bigg) \,\du -\dfrac{\log\pi}{2\pi}\,\widehat{h^{\pm}_\Delta}(0) \\ 
			& \ \ \ \ + h^{\pm}_\Delta\left(t-\frac{1}{2i}\right)+h^{\pm}_\Delta\left(t+\frac{1}{2i}\right)- \frac{1}{\pi}\sum_{n\geq 2}\dfrac{\Lambda(n)}{\sqrt{n}} \,\widehat{h^{\pm}_{\Delta}}\left(\dfrac{\log n}{2\pi}\right) \cos(t \log n).
		\end{align*}
		Combining Lemmas \ref{1_57pm}, \ref{1_57pm2}, \ref{1_57pm3}, and \eqref{II)}: $\widehat{h^{\pm}_\Delta}(0)=\pi\big(1+\frac{2}{\pm e^{\pi\Delta}-1}\big)$, we obtain, for $t>0$ and $\Delta\geq \frac{1}{2}$,
		\begin{align*}
			\displaystyle\sum_{\gamma}h^{\pm}_\Delta(t-\gamma)  = & 
			\bigg(\frac{1}{2} +\frac{1}{\pm e^{\pi\Delta}-1} \bigg)\log\left(\frac{t}{2\pi}\right) +O^*\left( \dfrac{4.7}{t^2}\right) \\ 
			& \,\,\,\, +O^*\left(\dfrac{e^{\pi\Delta}}{c^{\pm}_\Delta}\left(2\pi\Delta - (1+\gamma) + \dfrac{\log 2\pi}{e^{2\pi\Delta}} + \dfrac{2|B|}{e^{\pi\Delta}}-\frac{1}{6e^{6\pi\Delta}}\right)\right).
		\end{align*}
		Substituting this identity into \eqref{10_23am}, 
		we obtain for $t>0$ and $\Delta\geq \frac{1}{2}$,
		\begin{align} \label{7_42pm}
		\!\!\!	\pm \re\dfrac{\zeta'}{\zeta}(1+it)\leq &  \dfrac{\log t - \log 2\pi}{ e^{\pi\Delta}\mp 1}+\frac{e^{2\pi\Delta}}{(e^{\pi\Delta}\mp 1)^2}\left(2\pi\Delta - (1+\gamma) + \dfrac{\log 2\pi}{e^{2\pi\Delta}} + \dfrac{2|B|}{e^{\pi\Delta}}\right) + \left(\dfrac{6.5}{t^2}-\dfrac{1}{6c^{\pm}_\Delta e^{5\pi\Delta}}\right).
		\end{align} 

		Since our goal is to minimize the principal term of $ \re({\zeta'}/{\zeta})\left(1+\ie t\right)$, this term must be $2\pi\Delta$.
		Therefore, in order to optimize $\Delta$, we must ensure that $e^{-\pi\Delta}\log t=o(2\pi\Delta)$.
		Let $\pi\Delta=(1-\epsilon(t))\log\log t$. From the previous bound, it follows that $\epsilon \to 0$ as $t \to\infty$, hence $2\log\log t$ is the principal term. In order to minimize the constant term, we choose $\epsilon(t) = \frac{\log2}{\log\log t}$. This implies that $e^{\pi\Delta}=\frac{\log t}{2}$. Then, we rewrite
		$$
		\frac{\log t - \log 2\pi}{e^{\pi\Delta} \mp  1} = 2  + \dfrac{\pm 2 - \log 2\pi}{e^{\pi\Delta} \mp 1},
		$$
		and 
		$$
		\frac{e^{2\pi\Delta}}{(e^{\pi\Delta}\mp 1)^2}(2\pi\Delta - (1+\gamma)) = 2\pi\Delta -1 -\gamma -  \left(\dfrac{1\mp 2e^{\pi\Delta}}{(e^{\pi\Delta} \mp 1)^2}\right)\,(2\pi\Delta -1 -\gamma).
		$$
		Thus, in \eqref{7_42pm}, for $t\geq e^{18}$ (that is, $\Delta\geq 0.699\ldots$), the last term in parentheses on the right-hand side of \eqref{7_42pm} is negative; therefore,
		\begin{equation*} 
			\pm \re\frac{\zeta'}{\zeta}\left(1+\ie t\right)  \leq 2\pi\Delta + 1 -\gamma + \varepsilon_{\pm}(\Delta),
		\end{equation*}
		where
		\begin{align*}
			\varepsilon_{\pm}(\Delta) & = \frac{\pm 2-\log 2\pi}{e^{\pi\Delta} \mp  1} -\dfrac{(1\mp 2e^{\pi\Delta})(2\pi\Delta -1 -\gamma)}{(e^{\pi\Delta}\mp 1)^{2}}+ \dfrac{2|B|e^{\pi\Delta}+\log2\pi}{(e^{\pi\Delta} \mp 1)^2} \\
			& = \dfrac{2+(1\mp2)\gamma-\log8\pi^2\pm 4\pi\Delta}{e^{\pi\Delta} \mp 1}+\dfrac{\pm 2-1+(\pm1-1)\gamma + \log2\pi\mp\log 4\pi +2\pi\Delta}{(e^{\pi\Delta}\mp 1)^{2}},
		\end{align*}
	where we used \eqref{3_35pm}. Since $t\geq e^{18}$, we have $\pi\Delta \geq \log 9$, and numerically\footnote{It can be shown that the functions $\Delta\mapsto 2e^{\pi\Delta}\varepsilon_{\pm}(\Delta)\mp8\log({2e^{\pi\Delta}})$ are decreasing for $\pi\Delta\geq \log 9$.}, we can show that 
		$$\varepsilon_{\pm}(\Delta)\leq \pm\dfrac{8\log(2e^{\pi\Delta})}{2e^{\pi\Delta}}\mp \frac{\eta^{\pm}}{2e^{\pi\Delta}},$$ 
		where $\eta^{+}=8.6544$ and $\eta^{-}=6.9856$. Recalling that $e^{\pi\Delta}= \frac{\log t}{2}$, we finally obtain, for $t\geq e^{18}$,
		\begin{align*}
			\pm\re\frac{\zeta'}{\zeta}\left(1+\ie t\right) \leq 2\log\log t +1-\gamma-\log4 \pm \frac{8\log\log t}{\log t}\mp\dfrac{\eta^\pm}{\log t}.
		\end{align*}
This completes the proof of Theorem \ref{thm:Rezeta'/zeta}.

					\section{Lemmas related to the work of Lamzouri, Li, and Soundararajan} In this section we establish several lemmas that will be useful to prove Theorem \ref{thm:zeta}. These results were either established earlier in \cite{Sound} or are modifications that are partially similar, with proofs following analogous arguments.
					
					\begin{lemma} \label{logzeta}
						Assume the Riemann hypothesis. For any $t \geq2$ and $x \geq 2$ we have
						\begin{equation*}
							\begin{split}
								\log |\zeta(1+it)| & = \re\sum_{n\leq x}\frac{\Lambda(n)}{n^{1+it}\log n}\frac{\log(x/n)}{\log x} \, -\frac{1}{\log x}\re\frac{\zeta'}{\zeta}(1+it) \, \\ & \ \ \ \  + O^*\bigg(\frac{1}{\sqrt{x}\log^2x}\sum_{\rho}\frac{1}{|\rho-it|^2}\bigg)+ O^*\bigg(\frac{2}{t^2\log^2x}\bigg).
							\end{split}
						\end{equation*}
					\end{lemma}
					\begin{proof} The proof follows \cite[Lemma 2.5]{Sound} closely. Let $\sigma\geq 1$. By Perron's formula, we have
						\begin{equation*}
							\frac{1}{2\pi i}\int_{2-i\infty}^{2+i\infty} -\frac{\zeta'}{\zeta}(s+\sigma+it)\frac{x^s}{s^2}\ds = \sum_{n\leq x}\frac{\Lambda(n)}{n^{\sigma+it}}\log\left(\dfrac{x}{n}\right).
						\end{equation*}
						Since the integrand has poles at $s=1-\sigma-it$ (order $1$), $s=0$ (order $2$), $s=\rho-\sigma-it$ (order $1$ with $\rho$ non-trivial zero of $\zeta$), and $s=-2n-\sigma-it$, $n\in\N$ (order $1$), we find that, after shifting the line of integration to the left, the above integral also equals
						\begin{equation*}
							\frac{x^{1-\sigma-it}}{(\sigma-1+it)^2} - \frac{\zeta'}{\zeta}(\sigma+it)\log x - \bigg(\frac{\zeta'}{\zeta}\bigg)'(\sigma+it) -
							\sum_{\rho}\dfrac{x^{\rho-\sigma-it}}{(\rho-\sigma-it)^2} - \sum_{n=1}^\infty\dfrac{x^{-2n-\sigma-it}}{(2n+\sigma+it)^2}.
						\end{equation*}
						Therefore, using that $|\rho-\sigma-it|\geq |\rho-it|$ due to RH and rearranging the terms, we obtain
						\begin{equation*}
							\begin{split}
								- \frac{\zeta'}{\zeta}(\sigma+it) & = \frac{1}{\log x}\sum_{n\leq x}\frac{\Lambda(n)}{n^{\sigma+it}}\log\left(\dfrac{x}{n}\right) + \frac{1}{\log x}\bigg(\frac{\zeta'}{\zeta}\bigg)'(\sigma+it) \\ & \ \ \ \ \, + O^*\bigg(\frac{1+2x^{-3}}{t^2\log x}x^{1-\sigma}\bigg) +O^*\bigg(\frac{x^{\frac{1}{2} -\sigma}}{\log x}\sum_{\rho}\frac{1}{|\rho-it|^2}\bigg).
							\end{split}
						\end{equation*}
						The lemma follows by integrating over $\sigma$ from 1 to $\infty$ and taking real parts on both sides.
					\end{proof}

					\begin{lemma} \label{12_01pm}
						Assume the Riemann hypothesis. For all $x\geq e$, we have
						$$
						\sum_{n\leq x}\frac{\Lambda(n)}{n\log n}\frac{\log(x/n)}{\log x} \leq \log\log x + \gamma -1 + \dfrac{\gamma}{\log x} + \dfrac{2|B|}{\sqrt{x}\log^2x}
						$$
						where $B$ is defined in \eqref{3_35pm}.
					\end{lemma}
					\begin{proof}
						See \cite[Lemma 2.6]{Sound} for the proof. Here we may omit the term $\frac{\theta}{3x^3\log^2x}$ appearing in \cite[Lemma 2.6]{Sound}, since its contribution becomes negative later in the proof, thereby yielding our inequality.
					\end{proof}
					
					\begin{lemma} \label{12_01pm2}
						Assume the Riemann hypothesis. For all $x>1$, we have
						\begin{align*}
							\sum_{n\leq x}\frac{\Lambda(n)}{n}\left(1-\frac{n}{x}\right) \leq \log x - (1+\gamma) + \dfrac{\log 2\pi}{x} + \dfrac{2|B|}{\sqrt{x}}-\dfrac{1}{6x^3}.
						\end{align*}
						where $B$ is defined in \eqref{3_35pm}.
					\end{lemma}
					\begin{proof}
						See \cite[Lemma 2.4]{Sound} for the proof. Keeping only the first term of the series $\sum_{n=1}^\infty\frac{x^{-2n-1}}{2n(2n+1)}$, leads to the stated inequality.
					\end{proof}

					\begin{lemma} \label{12_31am} For $x\geq 81$ and $t\in\R$ we have 
						\begin{align*}
							\begin{split}
								\re\sum_{n\leq x}\frac{\Lambda(n)}{n^{it}}\left(\frac{1}{n\log n}-\frac{1}{x\log x}\right) \geq & -\sum_{n \leq x}\frac{\Lambda(n)}{n\log n}\frac{\log(x/n)}{\log x} - \frac{1}{\log x}\sum_{n\leq x}\frac{\Lambda(n)}{n}\left(1-\frac{n}{x}\right)  \\ & + \log\zeta(2) - \dfrac{0.249}{\sqrt{x}} -\dfrac{2}{\sqrt{x}\log x}.
							\end{split}
						\end{align*}
					\end{lemma}
					\begin{proof}
					We first prove that for $x\geq 81$,
						\begin{equation} \label{2_24am}
							\re\sum_{n\leq x}\frac{\Lambda(n)}{n^{it}}\left(\frac{1}{n\log n}-\frac{1}{x\log x}\right) \geq \sum_{p^k\leq x} \Lambda(p^k)(-1)^k\left(\frac{1}{p^k\log p^k}-\frac{1}{x\log x}\right).
						\end{equation}
						The proof of this estimate for $x\geq 100$ follows from \cite[Lemma 5.1]{Sound}, since its argument only requires the function $\chi$
						to be completely multiplicative,
						and in our case, the quantity playing this role is $n^{-it}$. We now explain how to extend the result to the range $x\geq 81$. As noted in \cite{Sound}, we consider the contribution of the powers of a fixed prime $p\leq x$ to both sides of \eqref{2_24am} and take the difference. For a fixed prime $p\geq 2$, writing $p^{-it}=-e^{i\theta}$, the difference between the contributions to the left-hand side and the right-hand side is 
						\begin{align} \label{1_59pm} 
							(\log p) \sum_{k\leq \frac{\log x}{\log p}}(-1)^{k-1}(1-\cos(k\theta))\left(\frac{1}{p^k\log p^k}-\frac{1}{x\log x}\right). \end{align} 
						When $p\geq 3$, we use the bound $1-\cos(k\theta)\leq k^2(1-\cos\theta)$. In \eqref{1_59pm}, separate the term $k=1$ and discard the terms with odd $k$, as well as other positive terms. The resulting sum is then bounded below by \begin{align} \label{2_13pm} 
							(\log p)(1-\cos\theta)\left(\frac{1}{p\log p}-\frac{1}{x\log x}-\sum_{l\leq \frac{\log x}{2\log p}}\dfrac{(2l)^2}{p^{2l}\log p^{2l}}\right). \end{align} If $p^2>x$, the sum in \eqref{2_13pm} is empty, which implies that the expression \eqref{2_13pm} (and hence also \eqref{1_59pm}) is nonnegative. If $p^2 \leq x$, the expression in \eqref{2_13pm} is bounded below by the same expression obtained by completing the sum over all $l\geq 1$, namely $$(\log p)(1-\cos\theta)\left(\dfrac{1}{p\log p}-\dfrac{1}{x\log x} - \dfrac{2p^2}{(p^2-1)^2\log p}\right)$$ which is again nonnegative.
						
						 When $p=2$, a different argument is required. Assume $81\leq x\leq 100$. Since $6<\frac{\log x}{\log p} <7$, the expression in \eqref{1_59pm} is bounded below by $\log 2$ times the trigonometrical polynomial $$\sum_{j\in\{1,3,5\}} (1-\cos(j\theta))\left(\dfrac{1}{2^{j}\log 2^{j}}-\dfrac{1}{81\log 81}\right) - \sum_{j\in\{2,4,6\}} (1-\cos(j\theta))\left(\dfrac{1}{2^{j}\log 2^{j}}-\dfrac{1}{100\log 100}\right).$$ Since this trigonometric polynomial is even and has period $2\pi$, one can verify numerically that its minimum value is $0$. This completes the proof of \eqref{2_24am}.

Now, the right-hand side of \eqref{2_24am} is exactly (\cite[Eq. (5.3)]{Sound}) 
						$$
						-\sum_{n \leq x}\frac{\Lambda(n)}{n\log n}\frac{\log(x/n)}{\log x} - \frac{1}{\log x}\sum_{n\leq x}\frac{\Lambda(n)}{n}\left(1-\frac{n}{x}\right) + 2 \sum_{m^2\leq x}\Lambda(m)\left(\dfrac{1}{m^2\log m^2}-\dfrac{1}{x\log x}\right).
						$$
						The third term above is $\log\zeta(2) -Q(x)$, where
						$$
						Q(x)=\sum_{m>\sqrt{x}}\dfrac{\Lambda(m)}{m^2 \log m} + 2\sum_{m\leq\sqrt{x}}\dfrac{\Lambda(m)}{x\log x}.
						$$
						It remains to prove that for $x\geq 81$
						\begin{align} \label{5_18pm}
							Q(x)\leq \left(0.249+ \dfrac{2}{\log x}\right)\dfrac{1}{\sqrt{x}}.
						\end{align}
						We first verify that \eqref{5_18pm} holds for $x\ge (73.2)^2$. Applying integration by parts together with the conditional bound $\psi(u):=\sum_{n\leq u}\Lambda(n)\leq u+\frac{1}{8\pi}\sqrt{u}\log^2u$ which holds for $u\geq 73.2$, see \cite[Theorem 10]{SchoenfeldSharperRH}, 
							\begin{equation*}
							\begin{split}
								Q(x) & =\int_{\sqrt{x}}^{\infty}\psi(u)\left(\dfrac{2}{u^3\log u}+\dfrac{1}{u^3\log^2u}\right) \du  \\
								& \leq \int_{\sqrt{x}}^{\infty}\left(\dfrac{1}{u^2\log u}+\dfrac{1}{u^2\log^2 u}\right)\du+\int_{\sqrt{x}}^{\infty}\left(\dfrac{1}{u^2\log u}+\dfrac{2\log u+1}{8\pi u^{{5}/{2}}}\right)\du \\
								& = \frac{2}{\sqrt{x}\log x}+\int_{\sqrt{x}}^{\infty}\left(\dfrac{1}{u^2\log u}+\dfrac{2\log u+1}{8\pi u^{{5}/{2}}}\right)\du.
							\end{split}
						\end{equation*}

						Let $f(u)$ denote the integrand on the right-hand side of the above expression. It is straightforward\footnote{This can be shown by introducing the function $G(y)=y\int_{y}^\infty f(u)\du$, proving that $G''(y)\geq 0$ for $y\geq 73.2$, and using the fact that $\lim_{y\to +\infty}G'(y)=0$.} to check that $\sqrt{x}\int_{\sqrt{x}}^\infty f(u)\du$ is decreasing for $x\geq (73.2)^2$. Hence $\sqrt{x}Q(x)-\frac{2}{\log x}\leq 0.229$, which implies that \eqref{5_18pm} holds in this range. On the other hand, for $81\leq x< (73.2)^2$, we see that
						\begin{align*}
							Q(x) - \dfrac{2}{\sqrt{x}\log x} & =\log\zeta(2) - \sum_{m\leq \sqrt{x}}{\Lambda(m)}\left(\dfrac{1}{m^2\log m }- \dfrac{2}{x\log x}\right)- \dfrac{2}{\sqrt{x}\log x} \\
							& = \left(\sqrt{x}\log\zeta(2)- \sqrt{x}\sum_{m\leq \sqrt{x}}{\Lambda(m)}\left(\dfrac{1}{m^2\log m }- \dfrac{2}{x\log x}\right) - \dfrac{2}{\log x}\right)\dfrac{1}{\sqrt{x}}.
						\end{align*}
						Numerically, the supremum of the term in parentheses over $81\leq x\leq (73.2)^2$ is bounded by $0.249$. Therefore, \eqref{5_18pm} holds for all $x\geq 81$ and we are done. 
					\end{proof}
					

					\begin{lemma} \label{11_18pm}
						Assume the Riemann hypothesis. For any $t\geq2$ and $x\geq 2$ we have
						\begin{align*}
								\re\sum_{n\leq x}\frac{\Lambda(n)}{n^{1+it}}\left(1-\frac{n}{x}\right) + \re\frac{\zeta'}{\zeta}(1+it) & = -\dfrac{1}{x}\left(\log\left(\frac{t}{2\pi}\right) + {\re\frac{\zeta'}{\zeta}(1+it)}\right)  \\
								&  \,\,\,\,\,\,  + O^*\left(\frac{1}{\sqrt{x}}\sum_{\rho}\frac{1}{|\rho-it|^2}\right)+O^*\left(\frac{3}{2t^2}\right).
						\end{align*}
					\end{lemma}
					\begin{proof} The proof is partially based on \cite[Lemma 2.4]{Sound}. By Perron's formula, we have
						\begin{equation*}
							\frac{1}{2\pi i}\int_{2-i\infty}^{2+i\infty} -\frac{\zeta'}{\zeta}(s+it)\frac{x^{s-1}}{s(s-1)}\ds = \sum_{n\leq x}\frac{\Lambda(n)}{n^{1+it}}\left(1-\frac{n}{x}\right).
						\end{equation*}
						Since there are poles at $s=0$, $s=1$, $s=1-it$, $s=\rho-it$ and $s=-2n-it$, we move the line of integration to the left, and the above integral then equals
						\begin{equation*}
							\frac{1}{x}\frac{\zeta'}{\zeta}(it) - \frac{\zeta'}{\zeta}(1+it) +\frac{ix^{-it}}{t(1-it)} -
							\sum_{\rho}\dfrac{x^{\rho-1-it}}{(\rho-it)(\rho -1-it)} - \sum_{n=1}^\infty\dfrac{x^{-2n-1-it}}{(2n+it)(2n+1+it)}.
						\end{equation*}
						Since RH holds, we have that $|(\rho-it)(\rho-1-it)|\geq |\rho-it|^2$.
						Consequently,
						\begin{align} \label{12_03pm}
							\sum_{n\leq x}\frac{\Lambda(n)}{n^{1+it}}\left(1-\frac{n}{x}\right)+ \frac{\zeta'}{\zeta}(1+it) = \frac{1}{x}\frac{\zeta'}{\zeta}(it)  + O^*\left(\dfrac{1}{\sqrt{x}}\sum_{\rho}\frac{1}{|\rho-it|^2}\right)+ O^*\left(\dfrac{1+2x^{-3}}{t^2}\right).
						\end{align}
						Taking the real part of the logarithmic derivative in the functional equation of $\zeta(s)$ (see \cite[Eq. (10.27)]{MV}) and the Stirling's formula \eqref{3_53pm2} we get,
						\begin{align} \label{12_05am}
							\begin{split}
								\re \frac{\zeta'}{\zeta}(it) = \re \frac{\zeta'}{\zeta}(-it) & = -\re \frac{\zeta'}{\zeta}(1+it)+\log 2\pi -\re \frac{\Gamma'}{\Gamma}(1+it) \\
								& = -\re \frac{\zeta'}{\zeta}(1+it)-\log\left(\dfrac{t}{2\pi}\right) + O^*\left(\dfrac{1}{2t^2}\right).
							\end{split}
						\end{align}
						The result follows by taking real parts in \eqref{12_03pm}, inserting \eqref{12_05am}, and using that $x\geq 2$.
					\end{proof}
					
						\vspace{0.2cm}

					\section{Proof of Theorem \ref{thm:zeta}}

					\subsection{Proof of \eqref{3pm}: upper bound for $|\zeta(1+it)|$} Since the Riemann hypothesis holds, by Lemma \ref{4_04pm}, we have for $t>0$,
					\begin{align} \label{12_14pm2}
						\sum_{\rho}\frac{1}{|\rho-it|^2}  = \sum_{\gamma}\dfrac{1}{\frac{1}{4}+(t-\gamma)^2} = 2\sum_{\gamma}h(t-\gamma)
						\leq \log \left(\frac{t}{2\pi}\right) +2\re\frac{\zeta'}{\zeta}(1+it) + \frac{7}{2t^2}.
					\end{align}
					Then, rearranging the terms in Lemma \ref{logzeta}, we obtain, for $t\geq 2$ and $x\geq 2$,
					\begin{equation} \label{12_39pm}
						\begin{split}
							\log |\zeta(1+it)| & \leq  \sum_{n\leq x}\frac{\Lambda(n)}{n\log n}\frac{\log(x/n)}{\log x} +\bigg(\frac{1}{\log x}-\frac{2}{\sqrt{x}\log^2x}\bigg)\left(-\re\frac{\zeta'}{\zeta}(1+it)\right)  \\
							& \,\,\,\,\,\,\,\,\,+ \frac{1}{\sqrt{x}\log^2x}\log\left(\dfrac{t}{2\pi}\right)  +  \left(2+\dfrac{7}{2\sqrt{x}}\right)\frac{1}{t^2\log^2x}.
						\end{split}
					\end{equation}
					From Theorem \ref{thm:Rezeta'/zeta}, we deduce that, for $t\geq e^{18}$,
					\begin{align}  \label{1_07am}
						-\re\frac{\zeta'}{\zeta}\left(1+\ie t\right) \leq 2\log\log t +1-\gamma-\log4 - \frac{\eta^*\log\log t}{\log t},
					\end{align}
					where $\eta^*=5.583$. Combining Lemma \ref{12_01pm} and \eqref{1_07am} in \eqref{12_39pm}, we obtain, for $t\geq e^{18}$, $x\geq 13$,
					\begin{align*}
							\log |\zeta(1+it)| &  \leq \log\log x + \gamma
							+\frac{2\log\log t - \log x +1-\log4}{\log x} +\frac{\log t-4\log\log t+3\gamma-\log(\pi^2/2)}{\sqrt{x}\log^2 x} \\
							& \,\,\,\,\,\,\,\,\,\,\,\,\,-\bigg(\frac{1}{\log x}-\frac{2}{\sqrt{x}\log^2x}\bigg)\frac{\eta^{*}\log\log t}{\log t} +\frac{3}{t^2\log^2x}.
					\end{align*}
					In order for $\log\log t$ to appear as the leading term in Littlewood’s inequality, we require $\log x= 2(1+o(1))\log\log t$. We set $x=(c\log t)^2$, where $c>0$ will be specified later. For now, assume that $c<2$ and $(c\log t)^2\geq 13$. Since $2\log\log t - \log x = -2\log c$, it follows that
					\begin{align*} 
						\log |\zeta(1+it)|  \leq \log\log x + \gamma - \left({A(c)} - \dfrac{\beta(t)}{c\log x}+\dfrac{\eta^*\log\log t}{\log t}\right)\dfrac{1}{{\log x}},
					\end{align*}
					where $A(c)=2 \log c +\log 4-1$ and 
					$
					\beta(t)=1-\frac{4\log\log t}{\log t}+\frac{3\gamma-\log(\pi^2/2)}{\log t}+\frac{2\eta^{*}\log\log t}{\log^2 t}+\frac{6}{t^2}. 
					$ 
					For $t\geq e^{18}$ one clearly has $\beta(t)\leq 1- \theta(t)$, where $\theta(t)=\frac{3.332\log\log t}{\log t}$, which implies
					\begin{align} \label{2_37pm}
						\log |\zeta(1+it)|  \leq \log\log x + \gamma - \left({A(c)} - \dfrac{1-\theta(t)}{c\log x}+\dfrac{\eta^*\log\log t}{\log t}\right)\dfrac{1}{{\log x}}.
					\end{align}
					We now exponentiate both sides of \eqref{2_37pm} and apply the inequality $e^{-y}\leq 1-y+{y^2}/{2}$ for $y\geq 0$. Consequently, assuming that ${A(c)}+D(c,t,x)\geq 0$, where $D(c,t,x)=- \frac{1-\theta(t)}{c\log x}+\frac{\eta^*\log\log t}{\log t}$, we obtain 
					\begin{align} \label{5_43pm}
						\begin{split}
							|\zeta(1+it)| & \leq e^\gamma\bigg(\log x-A(c) -\dfrac{\eta^{*}\log\log t}{\log t}+\dfrac{1-\theta(t)}{c\log x}+\dfrac{(A(c)+D(c,t,x))^2}{2\log x}\bigg) \\
							&  = e^\gamma\bigg(\log x - A(c)  +\dfrac{2/c+A(c)^2}{2\log x}-\dfrac{\eta^{*}\log\log t}{\log t}-\dfrac{\theta(t)}{c\log x}+\dfrac{2A(c)D(c,t,x)+D(c,t,x)^2}{2\log x}\bigg).
						\end{split}
					\end{align}
					We then optimize the expression ${2}/{c}+A(c)^2$, choosing $c_0=1.0467$, which yields ${2}/{c_0}+A(c_0)^2=2.1388\ldots$. Observe that $t\geq e^{18}$ implies $x\geq 354.96\ldots>13$, and ${A(c_0)}+D(c_0,t,x)\geq 0$, so that the preceding assumptions are satisfied.
					
					Returning to the first line of \eqref{5_43pm}, we now use that ${A(c_0)}+D(c_0,t,x)\leq A(c_0)+\frac{5.112\log\log t}{\log t}$. Substituting $x=(c_0\log t)^2$, and noting that $\log x \geq 2\log\log t$, we obtain that $|\zeta(1+it)|$ is bounded by 
					\begin{align*} 
					& \,\,\,\,\,\,\,\,\, e^\gamma\bigg(\log x-A(c_0) -\dfrac{\eta^{*}\log\log t}{\log t}+\dfrac{1-\theta(t)}{c_0\log x}+\dfrac{1}{2\log x}\left(A(c_0)+\dfrac{5.112\log\log t}{\log t}\right)^2\bigg) \\
						&  \leq  e^\gamma\left(2\log\log t-\log 4 +1 +\dfrac{2.1388\ldots}{4\log\log t} +\left(\dfrac{5.112A(c_0)-3.332/c_0}{\log x }+\dfrac{(5.112)^2\log\log t}{2\log x\log t}-\eta^{*}\right)\dfrac{\log\log t}{\log t}\right).§
					\end{align*}
					Finally, recalling that $x\geq 354.96$, 
					we obtain, for $t\geq e^{18}$,
					\begin{equation*}
						|\zeta(1+it)|\leq 2e^\gamma\left(\log\log t - \log2+ \frac{1}{2}  +\frac{0.2674}{\log\log t}-\dfrac{2.676\log\log t}{\log t}\right).
					\end{equation*}
This proves \eqref{3pm}.
					

					
					

					\subsection{Proof of \eqref{3pm2}: lower bound for $|\zeta(1+it)|$}
					As in the previous section, we begin with Lemma \ref{logzeta}. Applying \eqref{12_14pm2}, we obtain, for $t\geq 2$ and $x\geq2$
					\begin{align*} 
							-\log |\zeta(1+it)| & \leq  -\re\sum_{n\leq x}\frac{\Lambda(n)}{n^{1+it}\log n}\frac{\log(x/n)}{\log x} +\bigg(\frac{1}{\log x}+\frac{2}{\sqrt{x}\log^2x}\bigg)\re\frac{\zeta'}{\zeta}(1+it)  \\
							& \,\,\,\,\,\,\,\,\,+ \frac{1}{\sqrt{x}\log^2x}\log\left(\dfrac{t}{2\pi}\right)  +  \left(2+\dfrac{7}{2\sqrt{x}}\right)\frac{1}{t^2\log^2x}.
					\end{align*}
					Next, observe that
					\begin{equation*}
						\re\sum_{n\leq x}\frac{\Lambda(n)}{n^{1+it}\log n}\frac{\log(x/n)}{\log x} = \re\sum_{n\leq x}\frac{\Lambda(n)}{n^{it}}\left(\frac{1}{n\log n}-\frac{1}{x\log x}\right) - \frac{1}{\log x}\re\sum_{n\leq x}\frac{\Lambda(n)}{n^{1+it}}\left(1-\frac{n}{x}\right).
					\end{equation*}
					Consequently, we obtain
					\begin{align*} 
						-\log |\zeta(1+it)| & \leq  -\re\sum_{n\leq x}\frac{\Lambda(n)}{n^{it}}\bigg(\dfrac{1}{n\log n}-\dfrac{1}{x\log x}\bigg) +\left(\re\sum_{n\leq x}\frac{\Lambda(n)}{n^{1+it}}\left(1-\frac{n}{x}\right) + \re\frac{\zeta'}{\zeta}(1+it)\right)\frac{1}{\log x}\\
						& \,\,\,\,\,\,\,\,\,\,\,\,+ \frac{1}{\sqrt{x}\log^2x}\log\left(\dfrac{t}{2\pi}\right)  +  \left(2+\dfrac{7}{2\sqrt{x}}\right)\frac{1}{t^2\log^2x}+\frac{2}{\sqrt{x}\log^2x}\,\re\frac{\zeta'}{\zeta}(1+it).
					\end{align*}
					Now, by Lemma \ref{12_31am}, Lemma \ref{11_18pm} and \eqref{12_14pm2}, we get for $t\geq 2$ and $x\geq 81$ in the above inequality,
					\begin{align*}
						-\log |\zeta(1+it)| & \leq  \sum_{n\leq x}\frac{\Lambda(n)}{n\log n}\dfrac{\log(x/n)}{\log x}-\log \zeta(2) +\frac{1}{\log x}\sum_{n\leq x}\frac{\Lambda(n)}{n}\left(1-\frac{n}{x}\right)   + \dfrac{0.249}{\sqrt{x}}+\dfrac{2}{\sqrt{x}\log x}  \\
						& \,\,\,\,\,\,\,\,\,+ \dfrac{1}{\sqrt{x}\log x}\left(1+\dfrac{1}{\log x}-\dfrac{1}{\sqrt{x}}\right)\log\left(\dfrac{t}{2\pi}\right)  + \dfrac{3}{t^2\log x} \\
						& \,\,\,\,\,\,\,\,\,+ \dfrac{2}{\sqrt{x}\log x}\bigg(1+\dfrac{1}{\log x}-\dfrac{1}{2\sqrt{x}}\bigg)\re \dfrac{\zeta'}{\zeta}(1+it).
					\end{align*}
							Applying Lemma \ref{12_01pm}, Lemma \ref{12_01pm2} and Theorem \ref{thm:Rezeta'/zeta},
							\begin{align*}
								-\log|\zeta(1+it)| & \leq \log \log x  + \gamma-\log\zeta(2)-\dfrac{1}{\log x}+ \dfrac{\log t}{\sqrt{x}\log x}\left(1+\dfrac{1}{\log x}-\dfrac{1}{\sqrt{x}}\right) \\ & \,\,\,\,\,\,\,\,\, - \dfrac{\log 2\pi}{\sqrt{x}\log x}\left(1+\dfrac{1}{\log x}-\dfrac{2}{\sqrt{x}}\right)+ \dfrac{0.249}{\sqrt{x}} + \dfrac{2}{\sqrt{x}\log x} +\dfrac{2|B|}{\sqrt{x}\log x}\left(1+\dfrac{1}{\log x}\right)   \\ & \,\,\,\,\,\,\,\,\,+ \dfrac{2}{\sqrt{x}\log x}\bigg(1+\dfrac{1}{\log x}-\dfrac{1}{2\sqrt{x}}\bigg)\left(2\log\log t+1-\gamma-\log 4+\dfrac{8\log\log t}{\log t}-\dfrac{\eta^+}{\log t}\right) \\
								& \,\,\,\,\,\,\,\,\, + \dfrac{3}{t^2\log x}- \dfrac{1}{6x^3\log x}.
							\end{align*}
							In order for $\log\log t$ to be the leading term in \eqref{3pm2}, we choose $\log x= 2(1+o(1))\log\log t$. The constant term is optimized by taking $x=\frac{\log^2t}{4}$. Note that $t\geq e^{18}$ implies $x\geq 81$. Then, the last line in the previous inequality is negative. Therefore, 
							\begin{equation} \label{6_37pm}
								-\log|\zeta(1+it)| \leq \log\log x  + \gamma- \log\zeta(2)+\frac{1}{\log x} + \frac{2}{\log^2x} + \frac{E(x)}{\sqrt{x}},
							\end{equation}
							where $b_1=4-\gamma-\log(8\pi^2)=-0.9461\ldots$, $b_2=4\log4-\eta^+=-3.1092\ldots$, and
							\begin{align*} 
								E(x)  =  2.249+\dfrac{b_1+2}{\log x}+\dfrac{b_1}{\log ^2x} + \dfrac{1}{\sqrt{x}}\left(3+\dfrac{b_2+3+\gamma+2\log2\pi}{\log x}+\frac{b_2}{\log^2x}-\dfrac{2}{\sqrt{x}}-\frac{b_2}{2\sqrt{x}\log x}\right).
							\end{align*}
							Since $E$ is decreasing, we see that $E(x)\leq 2.84$ for $x\geq 81$. 
							Taking the exponential of \eqref{6_37pm}, we obtain
							\begin{align} \label{4_27pm}
								\dfrac{1}{|\zeta(1+it)|}\leq \dfrac{e^\gamma}{\zeta(2)}\,\log x\cdot\exp\left(\frac{1}{\log x} + \frac{2}{\log^2x} + \frac{2.84}{\sqrt{x}}\right).
							\end{align}
							For $0\leq y\leq M$ we have that $\exp(y)=1+y+{y^2}/{2} + O^*\left(c_M\cdot y^3\right)$, where $c_M=(e^M-1-M-{M^2}/{2})/{M^3}$.
							Here, we take $y=y(x)=\frac{1}{\log x}+\frac{2}{\log^2x}+ \frac{2.84}{\sqrt{x}}\leq 0.647$, for $x\geq 81$, which implies
							$\exp(y)\leq 1+y+{y^2}/{2} + 0.198\,y^3$. 
							Therefore, for $x\geq 81$,
							\begin{align*} 
							\log x\cdot\exp\left(\frac{1}{\log x} + \frac{2}{\log^2x} + \frac{2.84}{\sqrt{x}}\right)\leq \log x+1+\dfrac{5}{2\log x}+\dfrac{R(x)}{\log^2(4x)},
							\end{align*}
							where
							\begin{align*}
							R(x)& =\left(\dfrac{2.84\log x}{\sqrt{x}}+\frac{\log x}{2}\left(y^2(x)-\dfrac{1}{\log^2x}\right)+0.198y^3(x)\log x\right)\log^2(4x).
						\end{align*}
We write $R(x)=R_1(x)+R_2(x)$, where
$$
R_1(x)
 =	\dfrac{2.84\log^2(4x)}{\sqrt{x}}\left(\log x + 1 +\dfrac{2}{\log x}\right)
 $$
 and
$$
R_2(x)=\left(\frac{\log(4x)}{\log x}\right)^2\left(2+\frac{2}{\log x}+\frac{2.84^2\log^3x}{2x}+0.198\left(1+\frac{2}{\log x}+\frac{2.84\log x}{\sqrt{x}}\right)^3\right)$$
are decreasing functions\footnote{One can prove that $R_1(x)$ is decreasing for $x\geq 81$ by setting $u=\log x$ and computing the logarithmic derivative.} for $x\geq 81$. In particular $R(x)\leq 81.107$ for $x\geq 81$. This implies that
							\begin{align*} 
	\log x\cdot\exp\left(\frac{1}{\log x} + \frac{2}{\log^2x} + \frac{2.84}{\sqrt{x}}\right) &\leq  \log x +1 + \frac{5}{2\log(4x)} + \frac{85.667}{\log^2(4x)}.
\end{align*}							
							Inserting this in \eqref{4_27pm}
							and replacing $x=\frac{\log^2t}{4}$, we obtain, for $t\geq e^{18}$,
							\begin{align*} 
								\left|\dfrac{1}{\zeta(1+it)}\right| \leq\frac{12e^\gamma}{\pi^2}\bigg(\log\log t -\log 2 +\frac{1}{2}+\frac{5}{8\log\log t} +\frac{10.7084}{(\log\log t)^2}\bigg).
							\end{align*}
This proves \eqref{3pm2}. The proof of Theorem \ref{thm:zeta} is thus complete.							

								\vspace{0.2cm}

							\section{Proof of Theorem \ref{thm:zeta'/zeta}}
							By Selberg's moment formula for the Riemann zeta-function \cite[Eq. (13.35)]{MV}, for any $x,y \geq 2$, $s\neq 1$ and $\zeta(s)\neq 0$,  we have
							\begin{equation}	\label{11_45pm}	
								\frac{\zeta'}{\zeta}(s) = -\sum_{n\leq xy}\frac{\Lambda_{x,y}(n)}{n^s} + \frac{1}{\log{y}}\sum_{\rho}\frac{x^{\rho-s}-(xy)^{\rho-s}}{\left(\rho-s\right)^{2}} + \frac{1}{\log{y}}\sum_{n=1}^{\infty}\frac{x^{-2n-s}-(xy)^{-2n-s}}{\left(2n+s\right)^{2}}-\frac{x^{1-s}-(xy)^{1-s}}{\left(1-s\right)^{2}\log y},
							\end{equation}
							where \[
							\Lambda_{x,y}(n) \de \left\{\begin{array}{ll}
								\Lambda(n), & 1\leq n\leq x, \\
								\Lambda(n)\frac{\log{({xy}/{n})}}{\log{y}}, & x<n\leq xy.
							\end{array}
							\right.
							\]
							Evaluating \eqref{11_45pm} at $s=\sigma+it$ with $\sigma\geq 1$ and $t\geq 0$ ($s\neq 1$), and assuming RH, we bound the sum of non-trivial zeros as 
							\begin{align*}
								\frac{1}{\log{y}}\left|\sum_{\rho}\frac{x^{\rho-(\sigma+it)}-(xy)^{\rho-(\sigma+it)}}{\left(\rho-(\sigma+it)\right)^{2}}\right|
								& \leq \frac{\big(y^{\frac{1}{2}}+1\big)(xy)^{-\frac{1}{2}}}{\log{y}}\sum_{\gamma}\frac{1}{\frac{1}{4}+(\gamma-t)^{2}},
							\end{align*}
							where we used that $|\rho-(\sigma+it)|^2\geq \frac{1}{4} + (\gamma-t)^2$. Similarly, the sum over trivial zeros satisfies
							\begin{align*}
								\frac{1}{\log{y}}\left|\sum_{n=1}^{\infty}\frac{x^{-2n-(\sigma+it)}-(xy)^{-2n-(\sigma+it)}}{\left(2n+(\sigma+it)\right)^{2}}\right|
								& \leq \frac{1}{(9+t^2)\log{y}}\sum_{n=1}^{\infty}\left(x^{-2n-1}+(xy)^{-2n-1}\right) = \frac{c_{x,y}}{9+t^2},
							\end{align*}
							where $c_{x,y}=\Big(\frac{1}{x^3-x}+\frac{1}{(xy)^3-(xy)}\Big)\frac{1}{\log y}$. Since $x,y\geq 2$, $0<c_{x,y}< 0.3$. Therefore, in \eqref{11_45pm} we arrive at
							\begin{align}	\label{11_45pm2}	
								\begin{split}
									\!\!\!	\frac{\zeta'}{\zeta}(s) & = -\sum_{n\leq xy}\frac{\Lambda_{x,y}(n)}{n^s} + O^*\left(\frac{\big(y^{\frac{1}{2}}+1\big)(xy)^{-\frac{1}{2}}}{\log{y}}\sum_{\gamma}\frac{1}{\frac{1}{4}+(\gamma-t)^{2}}\right) + O^*\left(\frac{c_{x,y}}{9+t^2}\right)-\frac{x^{1-s}-(xy)^{1-s}}{\left(1-s\right)^{2}\log y}.
								\end{split}
							\end{align}
							Setting $s=1+it$, with $t>0$ in \eqref{11_45pm2}, and using \eqref{12_14pm2} along with $x, y\geq 2$, we obtain
							\begin{equation} \label{1_42am}	
								\left|\frac{\zeta'}{\zeta}(1+it)\right| \leq  \sum_{n\leq xy}\frac{\Lambda_{x,y}(n)}{n}  +  \frac{\big(y^{\frac{1}{2}}+1\big)(xy)^{-\frac{1}{2}}}{\log{y}}\left(\log \left(\frac{t}{2\pi}\right) +2\re\frac{\zeta'}{\zeta}(1+it) + \frac{7}{2t^2}\right) + \dfrac{3.2}{t^2}.
							\end{equation}
							Now, we take $s=\sigma>1$, and thus $t=0$, in \eqref{11_45pm2}, and order conveniently to get that
							\begin{align}	\label{11_45pm3}	
								\sum_{n\leq xy}\frac{\Lambda_{x,y}(n)}{n^\sigma}  & = -\frac{\zeta'}{\zeta}(\sigma)-\frac{x^{1-\sigma}-(xy)^{1-\sigma}}{\left(1-\sigma\right)^{2}\log y}+ O^*\left(\frac{\big(y^{\frac{1}{2}}+1\big)(xy)^{-\frac{1}{2}}}{\log{y}}\sum_{\gamma}\frac{1}{\frac{1}{4}+\gamma^{2}}\right)+ O^*\left(\dfrac{c_{x,y}}{9}\right).
							\end{align}
							Under RH, we have $\sum_{\gamma}\frac{1}{\frac{1}{4}+\gamma^{2}}= \sum_{\rho}\frac{1}{|\rho|^2}=2|B|$. On the other hand, we know that (\cite[Exercise 6, p. 354]{MV}),
							$$-\frac{\zeta'}{\zeta}(\sigma)= \frac{1}{\sigma-1} -\gamma + O(\sigma-1)\,\,\,\, \mbox{as} \,\,\,\,\sigma\to 1.
							$$
							Therefore, letting $\sigma\to 1^+$ in \eqref{11_45pm3} we find that\footnote{This is the key point of divergence from the approach in \cite{ChiSim}. There, in \cite[Eq. (3.6)]{ChiSim}, the sum under consideration is estimated via \cite[Lemma 6]{ChiSim} and the classical conditional bound of Schoenfeld for $\psi(x)$. Here, we obtain an alternative estimate for this sum, leading to a significantly smaller error term.}
							\begin{align}	 \label{1_40am}
								\sum_{n\leq xy}\frac{\Lambda_{x,y}(n)}{n}  & = \log x -\gamma +\dfrac{\log y}{2}+ O^*\left(\frac{2|B|\big(y^{\frac{1}{2}}+1\big)(xy)^{-\frac{1}{2}}}{\log{y}}\right)+ O^*\left(\dfrac{c_{x,y}}{9}\right).
							\end{align}
							Now, we choose the parameters $y=e^{2\lambda}$ and $xy=\log^2t$, with $\lambda>0$ (under the assumptions $x,y\geq 2$). Thus, inserting \eqref{1_40am} in \eqref{1_42am} we obtain 
							\begin{equation*}	
								\left|\frac{\zeta'}{\zeta}(1+it)\right| \leq   2\log\log t  - \gamma - \lambda  +  \frac{e^\lambda+1}{2\lambda}\frac{1}{\log t}\left(\log \left(\frac{t}{2\pi}\right) +2\re\frac{\zeta'}{\zeta}(1+it) + 2|B|+ \frac{7}{2t^2}\right)+\dfrac{c_{x,y}}{9} + \dfrac{3.2}{t^2}.
							\end{equation*}
							Using Theorem \ref{thm:Rezeta'/zeta}, \eqref{3_35pm} and rearranging terms, we see that for $t\geq e^{18}$,
							\begin{align}	 \label{8_09pm}
								\begin{split}
									\left|\frac{\zeta'}{\zeta}(1+it)\right| & \leq   2\log\log t  - \gamma - \lambda+  \frac{e^\lambda+1}{2\lambda} + \frac{2(e^\lambda+1)}{\lambda}\dfrac{\log\log t}{\log t} \\
									& \,\,\,\,\,\,\,\,\, \, +\frac{e^\lambda+1}{2\lambda}\left(\dfrac{4-\gamma-\log(128\pi^2)}{\log t} + \dfrac{16\log\log t}{\log^2t} -\frac{17.308}{\log^2 t}\right)+\dfrac{c_{x,y}}{9}+ \dfrac{3.2}{t^2}.
								\end{split}
							\end{align}
							We choose $\lambda_0=2.1862$ in order to minimize the constant term in the first line of \eqref{8_09pm}. Now, $x\geq 4.089$ and $xy\geq 324$, and then $\frac{c_{x,y}}{9}\leq \big(\frac{0.06}{x^3}+\frac{0.056}{(xy)^3}\big)\frac{1}{\lambda_0}=\big(\frac{0.06e^{6\lambda_0}+0.056}{\lambda_0}\big)\frac{1}{\log^6t}<\frac{13652}{\log^6t}$. Then, 
							$$
							\frac{e^{\lambda_0}+1}{2\lambda_0}\left(\dfrac{16\log\log t}{\log^2t} -\frac{17.308}{\log^2 t}\right)+\frac{c_{x,y}}{9}+ \dfrac{3.2}{t^2} < \dfrac{3.648}{\log t}.
							$$
							for $t\geq e^{18}$. Inserting this in \eqref{8_09pm} we obtain 
							\begin{equation*}	
								\left|\frac{\zeta'}{\zeta}(1+it)\right|  \leq   2\log\log t  -\gamma  + 0.0784+\dfrac{9.0581\log\log t}{\log t} - \dfrac{4.773}{\log t}.
							\end{equation*} 
							This implies the desired result.

\begin{thebibliography}{9999}
								
								
								\bibitem{CChaM}
								E. Carneiro, V. Chandee and M. B. Milinovich, 
								\newblock Bounding $S(t)$ and $S_1(t)$ on the Riemann hypothesis,
								\newblock Math. Ann. 356 (2013), no. 3, 939--968.
								
								
								\bibitem{CChiM}  E. Carneiro, A. Chirre and M. B. Milinovich,
								\newblock Bandlimited approximations and estimates for the Riemann zeta-function,
								\newblock Publ. Mat. 63 (2019), no. 2, 601--661.
								
								\bibitem{CLV} E. Carneiro, F. Littmann and J. D. Vaaler, 
								\newblock Gaussian subordination for the Beurling-Selberg extremal problem, 
								\newblock Trans.
								Amer. Math. Soc. 365 (2013), no. 7, 3493–3534.
								
								\bibitem{CMi} E. Carneiro and M. B. Milinovich, 
								\newblock On Littlewood's estimate for the modulus of the zeta function on the critical line, 
								\newblock Proceedings of the International Conference Constructive Theory of Functions, Lozenets (2023), pp. 1-16, BAS, Sofia, 2024.
								
								
								\bibitem{CS}
								V. Chandee and K. Soundararajan,
								\newblock Bounding $|\zeta(\frac{1}{2}+it)|$ on the Riemann hypothesis,
								\newblock Bull. Lond. Math. Soc. 43 (2011), no. 2, 243--250.
								
								
								
								
								\bibitem{chirreGoncalves}
								A. Chirre and F. Gon\c{c}alves,
								\newblock Bounding the log-derivative of the zeta-function,
								\newblock  Math. Z. 300 (2022), no. 1, 1041--1053.
								
								\bibitem{ChiSim}  A. Chirre, M. V. Hagen and A.~Simoni\v{c} 
								\newblock Conditional estimates for the logarithmic derivative of Dirichlet $L$-functions,
								\newblock Indag. Math. (N.S.) 35 (2024), no. 1, 14--27.
								
								
								\bibitem{GG}
								D. A. Goldston and S. M. Gonek,
								\newblock A note on $S(t)$ and the zeros of the Riemann zeta-function,
								\newblock Bull. Lond. Math. Soc. 39 (2007), no. 3, 482--486.
								
								\bibitem{HaraldG} 
								H. A. Helfgott, 
								\newblock The ternary Goldbach problem,
								\newblock Second preliminary version. To appear in Ann. of Math. Studies, available at  {\url{https://webusers.imj-prg.fr/~harald.helfgott/anglais/book.html}}.
								
								
								
								
								\bibitem{Sound}
								Y. Lamzouri, X. Li and K. Soundararajan,
								\newblock Conditional bounds for the least quadratic non-residue and related problems,
								\newblock Math. Comp. 84 (2015), no. 295, 2391--2412.
								
								
								\bibitem{Langu}
								A. Languasco and T. Trudgian, Uniform effective estimates for $|L(1,\chi)|$, J. Number Theory 236 (2022), 245--260.
								
						
								
								\bibitem{Li1}
								J. E. Littlewood, 
								\newblock On the Riemann zeta-function,
								\newblock Proc. London Math. Soc. (2) 24 (1925),
								175–201.
								
								\bibitem{Li2}
								J. E. Littlewood, 
								\newblock On the function $1/\zeta(1 + it)$, 
								\newblock Proc. London Math. Soc. 27 (1928), 349–357.
								
								
								\bibitem{Lum} A. Lumley,
								\newblock Explicit bounds for $L$-functions on the edge of the critical strip, 
								\newblock J. Number Theory 188 (2018), 186–209.
								
								
								
								\bibitem{MV}
								H. L. Montgomery and R. C. Vaughan,
								\newblock {\it Multiplicative Number Theory: I. Classical Theory},
								\newblock Cambridge Studies in Advanced Mathematics 97, Cambridge University Press, 2006.
								
								\bibitem{SimonicPalo}
								N. Palojärvi and A.~Simoni\v{c},
								\newblock Conditional estimates for $L$-functions in the Selberg class,
								\newblock J. Number Theory 285 (2026), 135--193.
								
								
								
								\bibitem{SchoenfeldSharperRH}
								L.~Schoenfeld,
								\newblock Sharper bounds for the Chebyshev functions $\theta(x)$ and $\psi (x)$. {II},
								\newblock Math. Comp. 30 (1976), no.~134, 337--360.
								
								
								
								
								
								
								
								
								
								
							\end{thebibliography}
								\vspace{0.2cm}

							\section*{Acknowledgements}
							AC and BM were funded by the Vicerrectorado de Investigación (VRI) at PUCP through grant DGI-2025-PI1273. Part of this work was carried out as part of BM’s undergraduate thesis at the Pontifical Catholic University of Peru (PUCP). The authors would like to thank the anonymous referee for a careful reading of the manuscript and for several helpful comments and suggestions.

						\end{document}